\documentclass[12pt, reqno]{amsart}
\usepackage{ amsmath,amsthm, amscd, amsfonts, amssymb, graphicx, color}
\usepackage[bookmarksnumbered, colorlinks, plainpages,linkcolor=blue,urlcolor=blue,citecolor=blue]{hyperref}
\usepackage{setspace}
\usepackage{ulem}

\newtheorem{thm}{Theorem}[section]
\newtheorem{cor}[thm]{Corollary}
\newtheorem{lem}[thm]{Lemma}

\newtheorem{rem}[thm]{\bf{Remark}}

\newtheorem{exam}[thm]{Example}
\numberwithin{equation}{section}

\begin{document}

\title{generalized Cline's formula and Jacobson's lemma in a ring}

\author{Huanyin Chen}
\author{Marjan Sheibani}
\address{
School of Mathematics\\ Hangzhou Normal University\\ Hangzhou, China}
\email{<huanyinchenhz@163.com>}
\address{Farzanegan Campus, Semnan University, Semnan, Iran}
\email{<m.sheibani@semnan.ac.ir>}

\subjclass[2010]{16U99, 15A09, 47A11.} \keywords{Cline's formula; Jacobson's Lemma; Generalized Drazin inverse; Drazin inverse; ring.}

\begin{abstract}
We present new generalized Cline's formula and Jacobson's lemma for the g-Drazin inverse in a ring. These extend many known results, e.g.,
Chen and Abdolyousefi (Generalized Jacobson's Lemma in a Banach algebra,
Comm. Algebra, {\bf 49}(2021), 3263--3272), Yan and Zeng (The generalized inverses of the products of two elements in a ring, Turk. J. Math., {\bf 44}(2020), 1744--1756).\end{abstract}

\maketitle

\section{Introduction}

Let $R$ be an associative ring with an identity. The commutant of $a\in R$ is defined by $comm(a)=\{x\in
R~|~xa=ax\}$. The double commutant of $a\in R$ is defined by $comm^2(a)=\{x\in R~|~xy=yx~\mbox{for all}~y\in comm(a)\}$. An element $a\in R$ has g-Drazin inverse (i.e., generalized Drazin inverse) in case there exists $b\in R$ such that $b=bab, b\in comm^2(a), a-a^2b\in R^{qnil}.$ The preceding $b$ is unique if exists, we denote it by $a^{d}$. Here, $R^{qnil}=\{a\in R~|~1+ax\in R^{-1}~\mbox{for every}~x\in comm(a)\}$. An element $a\in R$ has Drazin inverse in case there exists $b\in R$ such that $b=bab, b\in comm^2(a), a-a^2b\in R^{nil}.$ The preceding $b$ is unique if exists, we denote it by $a^{D}$.

Let $a,b\in R$. As is well known, $ab$ is Drazin invertible if and only if so is $ba$. In this
case, $(ba)^D = b((ab)^D)^2a$. This equation is called Cline's formula. It plays an important
role in matrix theory and operator algebra.

In~\cite[Theorem 2.2]{LC}, Liao et al. generalized Cline's formula to the case of the generalized
Drazin invertibility. It was proved $ab\in R^d$ if and only if $ba\in R^d$. Generalized Cline's formula was also presented.

In ~\cite[Theorem 2.3]{L}. Cline's formula for a generalized Drazin inverse was extended to the case when $aba = aca$.

In ~\cite[Theorem 2.2]{CS2}, the authors presented a generalized Cline's formula for g-Drazin inverse under wider conditions $a(ba)^2=abaca$ $=acaba=(ac)^2a.$
Common spectral properties for bounded linear operators under the same conditions were studied by Zguitti (see~\cite{Z}).

If $A,B,C,D\in \mathcal{B}(X)$ satisfy $BAC=BDB$ and $CDB=CAC$, then $AC\in \mathcal{B}(X)^d$ if and only if $BD\in \mathcal{B}(X)^d$ (see~\cite[Theorem 3.4]{YZ2}).

In~\cite[Theorem 3.2]{M2}, Mosi\'c proved that $bd\in R^d$ if and only if $ac\in R^d$ under the conditions $bac=bdb$ and $cdb=cac$. The generalized Cline's formula was provided. This was also proved by Yan and Zeng in ~\cite[Theorem 2.2]{YF}.

Jacobson's Lemma states that $1-ab\in R^{-1}$ if and only if $1-ba\in R^{-1}$. It was extended to Drazin inverse in \cite{CH}. In~\cite[Theorem 2.3]{Z3}, Zhuang et al. proved that $1-ab\in R^d$ if and only if $1-ba\in R^d$ and a formula for the g-Drazin inverse of $1-ba$ in terms of the g-Drazin inverse of $1-ab$ was provided.

Corach et al. generalized Jacobson's lemma for g-Drazin inverse to the condition $aba=aca$ (see~\cite[Theorem 1]{C}).

In ~\cite[Theorem 2.2]{CS1}, the authors considered the wider conditions $a(ba)^2=abaca=acaba=(ac)^2a,$
and present an extension of Jacobson's lemma on the g-Drazin inverse for Banach algebras. This was extended from Banach algebras to rings by Ren and Jiang (see~\cite[Theorem 2.1]{R}).

If $A,B,C,D\in \mathcal{B}(X)$ satisfy $BAC=BDB$ and $CDB=CAC$, then $I-AC\in \mathcal{B}(X)^d$ if and only if $I-BD\in \mathcal{B}(X)^d$ (see~\cite[Theorem 4.4]{YZ2}).

In~\cite[Theorem 2.1]{M2}, Mosi\'c proved that $1-bd\in R^d$ if and only if $1-ac\in R^d$ under the conditions $bac=bdb$ and $cdb=cac$. This was also proved by Yan and Zeng (see~\cite[Theorem 3.1]{YF}).

In this paper, we establish new generalized Cline's formula and Jacobson's Lemma for g-Drazin inverse in a ring under the wider conditions. The preceding known theorems are thereby obtained as the special case of our results.

Throughout the paper, all rings are associative with an identity. We use $R^{-1}$ and $R^{nil}$ to denote the sets
of all units and all nilpotents of the ring $R$, respectively. $R^{D}$ and $R^{d}$ denote the sets of all Drazin and g-Drazin invertible elements in $R$. ${\Bbb C}$ stands for the field of all complex numbers.

\section{generalized Cline's Formula}

Let $a,b\in R$. As is well known, $ab\in R^{qnil}$ if and only if $ba\in R^{qnil}$ (see~\cite[Lemma 2.2]{L}). We begin with the following extension.

\begin{lem} Let $R$ be a ring, and let $a,b,c,d\in R$ satisfying
$$\begin{array}{c}
b(db)(ac)=b(db)^2;\\
c(ac)(db)=c(db)^2.
\end{array}$$ If $ac\in R^{qnil}$, then $bd\in R^{qnil}$.\end{lem}
\begin{proof} Let $x\in comm(bd)$. Then we have
$$\begin{array}{lll}
(acdbdx^5bdb)ac&=&acdbdx^5(bdbac)\\
&=&acdbdx^5(bdbdb)\\
&=&a(cdbdb)dbdx^5b\\
&=&a(cacdb)dbdx^5b\\
&=&ac(acdbdx^5bdb)
\end{array}$$
We infer that $acdbdx^5bdb\in comm(ac)$; hence, $1-(acdbd)(x^5bdbac)$ $=1-(acdbdx^5bdb)ac\in R^{-1}$. By using Jacobson's Lemma (see~\cite[Lemma 2.1]{L}),
we have $$\begin{array}{lll}
1-x^5bdbdbdbdbd&=&1-x^5bd(bdbdb)dbd\\
&=&1-x^5bd(bdbac)dbd\\
&=&1-x^5(bdbdb)acdbd\\
&=&1-x^5(bdbacac)dbd\\
&=&1-(x^5bdbacac)dbd\\
&\in& R^{-1}.
\end{array}$$ Then
$$\begin{array}{ll}
&(1-xbd)(1+xbd+x^2bdbd+x^3bdbdbd+x^4bdbdbdbd)\\
=&(1+xbd+x^2bdbd+x^3bdbdbd+x^4bdbdbdbd)(1-xbd)\\
=&1-x^5bdbdbdbdbd\\
\in& R^{-1}.
\end{array}$$ Accordingly, $bd\in R^{qnil}$, as desired.\end{proof}

We come now to the main result of this paper.

\begin{thm} Let $R$ be a ring, and let $a,b,c,d\in R$ satisfying
$$\begin{array}{c}
b(ac)^2=b(ac)(db)=b(db)(ac)=b(db)^2;\\
c(ac)^2=c(ac)(db)=c(db)(ac)=c(db)^2.
\end{array}$$
Then $ac\in R^{d}$ if and only if $bd\in R^{d}$ and
$(bd)^{d}=b((ac)^{d})^2d$.\end{thm}
\begin{proof} $\Longrightarrow $ Suppose that $ac$ has g-Drazin inverse and $(ac)^{d}$ $=h$. Let $e=bh^2d$ and $t\in comm(bd)$. We check that
$$\begin{array}{lll}
ac(acdtbdbdbac)&=&(acdbdbdbd)(tbac)\\
&=&acdtbdbdbdbac\\
&=&(acdtbdbdbac)ac.
\end{array}$$
Hence, $acdtbdbdbac\in comm(ac)$, and so $(acdtbdbdbac)h=h(acdtbdbdbac).$ We compute that
 $$\begin{array}{lll}
 et&=&bh^7(ac)^5dt\\
 &=&bh^7(acdtbdbdbac)d\\
 &=&b(acdtbdbdbac)h^7d\\
 &=&(bacdb)dtbdbach^7d\\
 &=&(bdbdb)dtbdbach^7d\\
 &=&tbdbdbdtbdbach^7d\\
  &=&tb(ac)^5h^7d\\
 &=&tbh^2d\\
 &=&te.
 \end{array}$$
 This implies that $e\in comm^2(bd)$.

We easily check that
 $$\begin{array}{lll}
 e(bd)e&=&bh^2d(bd)bh^2d=b(h^5acacac)dbdbh^2d\\
 &=&bh^5ac(db)^4h^2d=bh^5(ac)^5h^2d=bh^2d=e.
 \end{array}$$
Let $p=1-(ac)h$. Then $(pa)c=ac-achac=ac-(ac)^2h\in R^{qnil}$.
Moreover, we have
  $$\begin{array}{lll}
  bd-(bd)^2e&=&bd-bdbdbh^2d\\
  &=&bd-b(dbdbac)h^3d\\
  &=&bd-b(ac)^3h^3d\\
  &=&b(1-ach)d\\
  &=&b(pd).
  \end{array}$$
We directly compute that
$$\begin{array}{rll}
b(pac)(pdb)&=&b(ac)(db)-b(ac)^da(cacdb)\\
&=&b(ac)^2-b(ac)^3(ac)^d=bp(ac)^2=b(pac)^2,\\
b(pdb)(pac)&=&b[1-ac(ac)^d]dbacp=bdbacp-b(ac)^dacdbacp\\
&=&bdbacp-b(ac)^dp(ac)^3=bdbacp=b(pac)^2,\\
b(pdb)^2&=&bpdb[1-ac(ac)^d]db=bpdbdb-bpdbac(ac)^ddb\\
&=&bpdbdb-b[1-ac(ac)^d]dbac(ac)^ddb=bpdbdb\\
&=&b[1-ac(ac)^d]dbdb=bdbdb-b(ac)^dacdbdb\\
&=&b(ac)^2-b(ac)^dac(ac)^2=bp(ac)^2=b(pac)^2.
\end{array}$$
Furthermore, we have
$$\begin{array}{rll}
c(pac)(pdb)&=&c[1-ac(ac)^d]acdb=cacdb-c(ac)^da(cacdb)\\
&=&c(ac)^2-c(ac)^d(ac)^3=cp(ac)^2=c(pac)^2,\\
c(pdb)(pac)&=&c[1-ac(ac)^d]dbacp=cdbacp-c(ac)^dacdbacp\\
&=&c(ac)^2p-c(ac)^dp(ac)^3=c(pac)^2,\\
c(pac)(pdb)&=&c[1-ac(ac)^d]acdb=cacdb-c(ac)^da(cacdb)\\
&=&c(ac)^2-c(ac)^d(ac)^3=cp(ac)^2=c(pac)^2.
\end{array}$$
Therefore $$\begin{array}{c}
b(pac)^2=b(pac)(pdb)=b(pdb)(pac)=b(pdb)^2;\\
c(pac)^2=c(pac)(pdb)=c(pdb)(pac)=c(pdb)^2.
\end{array}$$
In light of Lemma 2.1, $b(pd)\in R^{qnil}$. Therefore $bd$ has g-Drazin inverse $e$. That is,
$e=bh^2a=(bd)^{d},$ as desired.

$\Longleftarrow $ This is obvious by the symmetry and Clin's formula.\end{proof}

In the case that $c=b$ and $d=a$, we recover the Cline's formula for g-Drazin inverse (see~\cite[Theorem 2.2]{LC}). Moreover, we derive

\begin{cor} (see~\cite[Theorem 2.2]{CS2}) Let $R$ be a ring, and let $a,b,c\in R$ satisfying $(ac)^2a=acaba=abaca=a(ba)^2.$ Then $ac\in R^{d}$ if and only if
$ba\in R^{d}$ and $(ba)^{d}=b((ac)^{d})^2a$.\end{cor}\begin{proof} By hypothesis, we have
$$a(caca)=a(caba)=a(baca)=a(baba).$$ In view of Theorem 2.2, $ca\in R^d$ if and only if $ab\in R^d$ and
$(ab)^d=a((ca)^{d})^2b$. By using Cline's formula (see~\cite[Theorem 2.2]{LC}), we prove that
$ac\in R^d$ if and only if $ba\in R^d$. In this case,
$$\begin{array}{lll}
(ba)^d&=&b[(ab)^d]^2a\\
&=&b[a((ca)^{d})^2b][a((ca)^{d})^2b]a\\
&=&ba((ca)^{d})^4c(acaca)((ca)^{d})^4c(acaca)\\
&=&ba[(ca)^d]^8(ca)^6\\
&=&ba[(ca)^d]^3ca\\
&=&ba[(ca)^d]^2ca[(ca)^d]^2ca\\
&=&b((ac)^{d})^2a,
\end{array}$$ as asserted.\end{proof}

\begin{cor}(see~\cite[Theorem 2.3]{L}) Let $R$ be a ring, and let $a,b,c\in R$ satisfying $aba=aca$. Then $ac\in R^d$ if and only if $ba\in R^{d}$ and $(ba)^{d}=b((ab)^{d})^2a$.\end{cor}\begin{proof} This is clear by Corollary 2.3.\end{proof}

\begin{cor} (see~\cite[Theorem 3.2]{M2} and \cite[Theorem 2.2]{YF}) Let $R$ be a ring, and let $a,b,c,d\in R$ satisfying $$\begin{array}{c}
bac=bdb;\\
cac=cdb.
\end{array}$$ Then $ac\in R^{d}$ if and only if $bd\in R^{d}$ and
$(bd)^{d}=b((ac)^{d})^2d$.\end{cor}
\begin{proof} This is obvious by Theorem 2.2.\end{proof}

\section{Extensions in Banach algebras}

The aim of this section is to investigate the Cline's formula in a Banach algebra under more simpler conditions. We now derive

\begin{thm} Let $\mathcal{A}$ be a Banach algebra, and let $a,b,c,d\in \mathcal{A}$ satisfying $$\begin{array}{c}
b(ac)^2=b(db)^2;\\
c(ac)^2=c(db)^2.
\end{array}$$
Then $ac\in \mathcal{A}^{d}$ if and only if $bd\in \mathcal{A}^{d}$. In this case, $(bd)^{d}=b[(ac)^{d}]^2d$.\end{thm}
\begin{proof} $\Longrightarrow $ Let $aca=a^{'}, c=c^{'}, dbd=d^{'}$ and $b=b^{'}$. Then we have
$$\begin{array}{c}
b'a'c'=b(ac)^2=b(db)^2=b'd'b';\\
c'a'c'=c(ac)^2=c(db)^2=c'd'b'.
\end{array}$$ Since $ac\in \mathcal{A}^{d}$, it follows by~\cite[Corollary 2.2]{M1} that $a'c'=(ac)^2\in \mathcal{A}^d$. By virtue of Corollary 2.5,
$b'd'=(bd)^2\in \mathcal{A}^d$. According to ~\cite[Corollary 2.2]{M1}, $bd\in \mathcal{A}^d$. In this case, we have
$$\begin{array}{lll}
(bd)^{d}&=&[(bd)^2]^{d}bd\\
&=&(b'd')^dbd=b'[(a'c')^d]^2d'bd\\
&=&b[(ac)^d]^4(db)^2d\\
&=&b[(ac)^d]^5ac(db)^2d\\
&=&b[(ac)^d]^5ac(ac)^2d\\
&=&b[(ac)^d]^2d,
\end{array}$$ as desired.

$\Longleftarrow $ This is obvious by the symmetry and Cline's formula.\end{proof}

\begin{cor} Let $\mathcal{A}$ be a Banach algebra, and let $a,b,c\in \mathcal{A}$ satisfying $a(ca)^2=(ab)^2a.$
Then $ac\in \mathcal{A}^{d}$ if and only if $ba\in \mathcal{A}^{d}$. In this case, $(ba)^{d}=b[(ac)^{d}]^2a$.\end{cor}
\begin{proof} Since $ac\in \mathcal{A}^d$, by Cline's formula, $ca\in \mathcal{A}^d$. As $a(ca)^2=(ab)^2a$, by using Theorem 3.1,
$ab\in \mathcal{A}^d$. By using Cline's formula again, $ba\in \mathcal{A}^{d}$. In this case,
$$\begin{array}{lll}
(ba)^{d}&=&b\big[(ab)^d]^2a\\
&=&b[a((ca)^d)^2c]^2a\\
&=&b[(ac)^d]^2a.
\end{array}$$ This completes the proof.\end{proof}

\begin{cor} Let $R$ be a ring, and let $a,b,c,d\in R$ satisfying
$$\begin{array}{c}
b(ac)^2=b(db)^2;\\
c(ac)^2=c(db)^2.
\end{array}$$
Then $ac\in R^{D}$ if and only if $bd\in R^{D}$ and
$(bd)^{D}=b((ac)^{D})^2d$.\end{cor}
\begin{proof}  Since $ac\in R^D$, we have $ac\in R^d$. In view of Theorem 3.1,
$bd\in R^d$ and $(bd)^d=b[(ac)^D]^2d$. We check that
$$\begin{array}{lll}
bd-(bd)^2(bd)^d&=&bd-(bd)^2b[(ac)^D]^2d\\
&=&bd-b(db)^2[(ac)^D]^2d\\
&=&bd-b(ac)^2[(ac)^D]^2d\\
&=&b[1-(ac)(ac)^D]d.
\end{array}$$ Then
$$\begin{array}{lll}
[bd-(bd)^2(bd)^d]bdbd&=&b[1-(ac)(ac)^D]dbdbd\\
&=&b(db)^2d-b(ac)^Dac(db)^2d\\
&=&b(ac)^2d-b(ac)^Dac(ac)^2d\\
&=&b[ac-(ac)^2(ac)^D]acd.\\
\end{array}$$ Hence,
$$\begin{array}{lll}
[bd-(bd)^2(bd)^d]^3&=&b[ac-(ac)^2(ac)^D]acd[1-(bd)(bd)^d].\\
\end{array}$$ As $ac-(ac)^2(ac)^D$ is nilpotent, by induction, we verify that $bd-(bd)^2(bd)^d$ is nilpotent.
Therefore $bd\in R^D$ and $(bd)^D=(bd)^d$, as asserted.\end{proof}

\begin{exam}\end{exam} Let $R=M_2({\Bbb C})$. Choose $$\begin{array}{c}
a=
\left(
\begin{array}{cc}
0&1\\
0&0
\end{array}
\right), b=\left(
\begin{array}{cc}
1&0\\
0&0
\end{array}
\right),\\
c=\left(
\begin{array}{cc}
1&0\\
1&1
\end{array}
\right)\in M_2({\Bbb C}).
\end{array}$$ Then $a(ca)^2=0=(ab)^2a$, while $aca=\left(
\begin{array}{cc}
0&1\\
0&0
\end{array}
\right)\neq 0=aba$. In this case, $(ac)^D=\left(
\begin{array}{cc}
0&1\\
0&1
\end{array}
\right)$ and $(ba)^D=0$.\\

\section{generalized Jacobson's lemma for g-Drazin inverse}

In this section, we investigate new extension of Jacobson's lemma for generalized Drazin inverse. We have the following new characterization.

\begin{thm} Let $R$ be a ring, and let $a,b,c,d\in R$ satisfying
$$\begin{array}{c}
b(ac)^2=b(ac)(db)=b(db)(ac)=b(db)^2;\\
c(ac)^2=c(ac)(db)=c(db)(ac)=c(db)^2.
\end{array}$$
If $\alpha=1-bd\in R^{d}$, then $\beta=1-ac\in R^{d}$ and
$\beta^{d}=\big[1-acd\alpha^{\pi}\big(1-\alpha\alpha^{\pi}(1+bd+bdbd)\big)^{-1}bac\big](1+ac+acac)+acd\alpha^dbac$.\end{thm}
\begin{proof} $\Longrightarrow$ Let $p=\alpha^{\pi},x=\alpha^d$. Then $1-p\alpha (1+bd+bdbd)\in R^{-1}$. Let $${\scriptsize y=\big[1-acdp\big(1-p\alpha(1+bd+bdbd)\big)^{-1}bac\big](1+ac+acac)+acd\alpha^dbac.}$$ We shall prove that $\beta^d=y$.

Step 1. $y\beta y=y$. We see that
$$y\beta=1-(ac)^3-acdp[1-p\alpha (1+bd+bdbd)]^{-1}bac[1-(ac)^3]+acdxbac(1-ac).$$
We compute that
$${\scriptsize\begin{array}{ll}
&y\beta\\
=&1-[acacac-acdxbac(1-ac)]-acdp[1-p\alpha (1+bd+bdbd)]^{-1}bac[1-(ac)^3]\\
=&1-ac[dbac-dxbac(1-ac)]-acdp[1-p\alpha (1+bd+bdbd)]^{-1}[bac-bac(ac)^3\big]\\
=&1-ac[dbac-dx(bac-bacac)]-acdp[1-p\alpha (1+bd+bdbd)]^{-1}(bac-bacacdbac)\\
=&1-ac[dbac-dx(bac-bacac)]-acdp[1-p\alpha (1+bd+bdbd)]^{-1}(1-bacacd)bac\\
=&1-ac[dbac-dx(1-bd)bac]-acdp[1-p\alpha (1+bd+bdbd)]^{-1}[1-(bd)^3]bac\\
=&1-acdpbac-acdp[1-p\alpha (1+bd+bdbd)]^{-1}p\alpha (1+bd+bdbd)bac\\
=&1-acdp[1-p\alpha (1+bd+bdbd)]^{-1}\big[(1-p\alpha (1+bd+bdbd))+p\alpha (1+bd+bdbd)\big]bac\\
=&1-acdp[1-p\alpha(1+bd+bdbd)]^{-1}bac.\\
\end{array}}$$
Since $bacacdb=ba(cacdb)=bacdbdb=bdbdbdb=bd(bacac)=bdbacac$, we have $(bacacd)(bd)=(bd)(bacacd)$, and so $(bacacd)\alpha =\alpha (bacacd).$ Hence,
$(bacacd)x=x(bacacd),$ and then
$$\begin{array}{ll}
&acdp[1-p\alpha(1+bd+bdbd)]^{-1}bacacdxbac\\
=&acd[1-p\alpha(1+bd+bdbd)]^{-1}pxbacacdbac\\
=&0.
\end{array}$$ Therefore we have
$${\tiny\begin{array}{ll}
&y\beta y\\
=&y-acdp[1-p\alpha(1+bd+bdbd)]^{-1}bacy\\
=&y-acdp[1-p\alpha(1+bd+bdbd)]^{-1}bac(1+ac+acac)\\
+&acdp[1-p\alpha(1+bd+bdbd)]^{-1}bacacdp\big(1-p\alpha(1+bd+bdbd)\big)^{-1}bac(1+ac+acac)\\
=&y-acdp[1-p\alpha(1+bd+bdbd)]^{-1}(1+bd+bdbd)bac\\
+&acdp[1-p\alpha(1+bd+bdbd)]^{-1}bdbdbdp\big(1-p\alpha(1+bd+bdbd)\big)^{-1}(1+bd+bdbd)bac\\
=&y-acdp[1-p\alpha(1+bd+bdbd)]^{-1}(1+bd+bdbd)bac\\
+&acdp[1-p\alpha(1+bd+bdbd)]^{-2}(bd)^3(1+bd+bdbd)bac\\
=&y-acdp[1-p\alpha(1+bd+bdbd)]^{-2}\big[p-p\alpha(1+bd+bdbd)-p(bd)^3\big](1+bd+bdbd)bac\\
=&y.\\
\end{array}}$$

Step 2. $\beta-\beta y\beta\in R^{qnil}$.
By Step 1, $y=y\beta y$, and so $(1-y\beta)^2=1-y\beta$. Hence,
$$\begin{array}{lll}
\beta-\beta y\beta&=&\beta (1-y\beta )^2\\
&=&\beta acd(bd)^3p[1-p\alpha(1+bd+bdbd)]^{-2}bac\\
&=&(1-ac)acacacacdp[1-p\alpha(1+bd+bdbd)]^{-2}bac\\
&=&[acdbac-acacdbac]acdp[1-p\alpha(1+bd+bdbd)]^{-2}bac\\
&=&[acdbac-acdbdbac]acdp[1-p\alpha(1+bd+bdbd)]^{-2}bac\\
&=&acd\alpha bacacdp[1-p\alpha(1+bd+bdbd)]^{-2}bac\\
&=&acd\alpha bdbdbdp[1-p\alpha(1+bd+bdbd)]^{-2}bac.
\end{array}$$ Let $z\in comm(\beta -\beta y\beta)$. Then
$$\begin{array}{ll}
&zacd\alpha bdbdbdp[1-p\alpha(1+bd+bdbd)]^{-2}bac\\
=&acd\alpha bdbdbdp[1-p\alpha(1+bd+bdbd)]^{-2}bacz.
\end{array}$$
We will suffice to prove $1+acd\alpha bdbdbdp[1-p\alpha(1+bd+bdbd)]^{-2}bacz\in R^{-1}.$
Obviously, we check that $p=(bd)^3p[1-p\alpha (1+bd+bdbd)]^{-1}=(bd)^6p[1-p\alpha (1+bd+bdbd)]^{-2}.$
Hence, we get
$$\begin{array}{ll}
&(baczacdbdbdbd)\alpha p\\
=&bac[zacd\alpha bdbdbdp[1-p\alpha (1+bd+bdbd)]^{-2}](bd)^6\\
=&bac[zacd\alpha bdbdbdp[1-p\alpha (1+bd+bdbd)]^{-2}bac](db)^4d\\
=&bac[zacd\alpha bdbdbdp[1-p\alpha (1+bd+bdbd)]^{-2}bac]ac(db)^3d\\
=&bac[acd\alpha bdbdbdp[1-p\alpha (1+bd+bdbd)]^{-2}bacz]ac(db)^3d\\
=&bdbdbd\alpha bdbdbdp[1-p\alpha (1+bd+bdbd)]^{-2}baczac(db)^3d\\
=&\alpha (bd)^6p[1-p\alpha (1+bd+bdbd)]^{-2}baczac(db)^3d\\
=&\alpha p(baczacdbdbdbd).
\end{array}$$

Step 3. $y\in comm^2(\beta)$. Let $s\in comm(\beta)$. Then $s\beta=\beta s$, and so $s(ac)=(ac)s$.

Claim 1. $s(acdxbac)=(acdxbac)s.$ We easily check that
$$(bacsdbd)\alpha=bacsd\alpha bd=bacs\beta dbd=bac\beta sdbd =\alpha (bacsdbd).$$ Hence $(bacsdbd)x=x(bacsdbd)$, and then
$$\begin{array}{lll}
s(acdpbac)&=&sd(bd)^4p[1-p\alpha (1+bd)]^{-2}bac\\
&=&s(ac)^2dbdbdp[1-p\alpha (1+bd)]^{-2}bac\\
&=&d(bacsdbd)bdp[1-p\alpha (1+bd)]^{-2}bac\\
&=&dbdp[1-p\alpha (1+bd)]^{-2}(bacsdbd)bac\\
&=&dbdp[1-p\alpha (1+bd)]^{-2}bacs(ac)^3\\
&=&dbdp[1-p\alpha (1+bd)]^{-2}(bd)^3bacs\\
&=&d(bd)^4p[1-p\alpha (1+bd)]^{-2}bacs\\
&=&(acdpbac)s.
\end{array}$$ Since $sacdbac=s(ac)^3=(ac)^3s=acdbacs$, we have
$sacd\alpha xbac=acd\alpha xbacs,$ and so
$sacdxbac-sacdbdxbac=acdxbacs-acdbdxbacs.$

On the other hand, we have
$$\begin{array}{lll}
s(acdbdpbac)&=&sd(bd)^5p[1-p\alpha (1+bd)]^{-2}bac\\
&=&s(ac)^4dbdp[1-p\alpha (1+bd)]^{-2}bac\\
&=&dbdbd(bacsdbd)p[1-p\alpha (1+bd)]^{-2}bac\\
&=&dbdbdp[1-p\alpha (1+bd)]^{-2}(bacsdbd)bac\\
&=&dbdbdp[1-p\alpha (1+bd)]^{-2}bacs(ac)^3\\
&=&dbdbdp[1-p\alpha (1+bd)]^{-2}(bd)^3bacs\\
&=&dbd(bd)^4p[1-p\alpha (1+bd)]^{-2}bacs\\
&=&(acdbdpbac)s.
\end{array}$$ Since $sacdbdbac=s(ac)^4=(ac)^4s=acdbdbacs$, we have
$acdbd\alpha xbacs=sacdbd\alpha xbac.$ Then we have $$\begin{array}{lll}
acacdbd\alpha xbacs&=&ac(acdbd\alpha xbacs)\\
&=&ac(sacdbd\alpha xbac)\\
&=&sacacdbd\alpha xbac\\
&=&sacdbdbd\alpha xbac,
\end{array}$$
and so $acdbd(1+bd)\alpha xbacs=sacdbd(1+bd)\alpha xbac,$ and then
$acdbdxbacs-acdbd(bd)^2xbacs=sacdbdxbac-sacdbd(bd)^2xbac.$
One easily checks that
$$\begin{array}{lll}
acd(bd)^3xbacs&=&acdx(bd)^3bacs\\
&=&acdxb(ac)^4s\\
&=&acdx(bacsdbd)bac\\
&=&acd(bacsdbd)xbac\\
&=&(ac)^3sdbdxbac\\
&=&s(ac)^3dbdxbac\\
&=&sacd(bd)^3xbac.
\end{array}$$
Hence $acdbdxbacs=sacdbdxbac$, and so $s(acdxbac)=(acdxbac)s.$

Claim 2. $sacdp\big[1-p\alpha(1+bd+bdbd)\big]^{-1}bac(1+ac+acac)=acdp\big[1-p\alpha(1+bd+bdbd)\big]^{-1}bac(1+ac+acac)s$.
Set $t=acdp\big[1-p\alpha(1+bd+bdbd)\big]^{-1}bac(1+ac+acac).$ Then we check that
$$\begin{array}{lll}
st&=&sacdp\big[1-p\alpha(1+bd+bdbd)\big]^{-1}bac(1+ac+acac)\\
&=&sacd(bd)^3p\big[1-p\alpha(1+bd+bdbd)\big]^{-2}bac(1+ac+acac)\\
&=&acacacsdbdp\big[1-p\alpha(1+bd+bdbd)\big]^{-2}bac(1+ac+acac)\\
&=&acd(bacsdbd)p\big[1-p\alpha(1+bd+bdbd)\big]^{-2}bac(1+ac+acac)\\
&=&acdp\big[1-p\alpha(1+bd+bdbd)\big]^{-2}bacsdbdbac(1+ac+acac),\\
\end{array}$$
Also we have
$$\begin{array}{lll}
ts&=&acdp\big[1-p\alpha(1+bd+bdbd)\big]^{-1}bac(1+ac+acac)s\\
&=&acdp[1-p\alpha (1+bd+bdbd)]^{-2}(bd)^3bac(1+ac+acac)s\\
&=&acdp[1-p\alpha (1+bd+bdbd)]^{-2}bsacacacac(1+ac+acac)\\
&=&acdp[1-p\alpha (1+bd+bdbd)]^{-2}bsacdbdbac(1+ac+acac)\\
&=&acdp[1-p\alpha (1+bd+bdbd)]^{-2}bacsdbdbac(1+ac+acac)\\\\
\end{array}$$
Then $st=ts$; hence, $y\in comm^2(\beta)$. Therefore $y=\beta^d$, as desired.

$\Longleftarrow$ Since $1-ac\in R^d$, it follows by Jacobson's Lemma that $1-ca\in R^d$. Applying the preceding discussion, we obtain that $1-bd\in R^d$, as desired.
\end{proof}

\begin{cor} (~\cite[Theorem 2.1]{M2} and ~\cite[Theorem 2.2]{YF}) Let $R$ be a ring, and let $a,b,c,d\in R$ satisfying $bac=bdb, cac=cdb.$ Then $\alpha=1-bd\in R^{d}$ if and only if
$\beta=1-ac\in R^d$. In this case,
$$\beta^{d}=\big[1-acd\alpha^{\pi}\big(1-\alpha\alpha^{\pi}(1+bd+bdbd)\big)^{-1}bac\big](1+ac+acac)+acd\alpha^dbac.$$\end{cor}
\begin{proof} By hypothesis, we have
$$\begin{array}{c}
b(ac)^2=b(ac)(db)=b(db)(ac)=b(db)^2;\\
c(ac)^2=c(ac)(db)=c(db)(ac)=c(db)^2.
\end{array}$$ This completes the proof by Theorem 4.1.\end{proof}

\begin{cor} (~\cite[Theorem 2.2]{CS1} and \cite[Theorem 2.1]{R}) Let $R$ be a ring, and let $a,b,c\in R$ satisfying $(ac)^2a=acaba=abaca=a(ba)^2.$ Then $\alpha=1-ba\in R^{d}$ if and only if
$\beta=1-ac\in R^d$. In this case, $$\beta^{d}=\big[1-aca\alpha^{\pi}\big(1-\alpha\alpha^{\pi}(1+ba+baba)\big)^{-1}bac\big](1+ac+acac)+aca\alpha^dbac.$$\end{cor}
\begin{proof} We complete the proof by Theorem 4.1 and Jacobson's lemma (see~\cite[Theorem 2.3]{LC}).\end{proof}

\begin{thm} Let $R$ be a ring, and let $a,b,c,d\in R$ satisfying
$$\begin{array}{c}
b(ac)^2=b(ac)(db)=b(db)(ac)=b(db)^2;\\
c(ac)^2=c(ac)(db)=c(db)(ac)=c(db)^2.
\end{array}$$
Then $\alpha=1-bd\in R^{D}$ if and only if $\beta=1-ac\in R^{D}$ and
$\beta^{D}=\big[1-acd\alpha^{\pi}\big(1-\alpha\alpha^{\pi}(1+bd+bdbd)\big)^{-1}bac\big](1+ac+acac)+acd\alpha^Dbac$.\end{thm}
\begin{proof} $\Longrightarrow $ In view of Theorem 4.1, $\beta=1-bd\in R^{d}$ and
$\beta^{d}=\big[1-acd\alpha^{\pi}\big(1-\alpha\alpha^{\pi}(1+bd+bdbd)\big)^{-1}bac\big](1+ac+acac)+acd\alpha^Dbac$. It will suffice to prove
$\beta\beta^{\pi}\in R^{nil}$. Let $p=\alpha^{\pi}$. As in the proof of Theorem 4.1, we have
$$\begin{array}{lll}
\beta\beta^{\pi}&=&acd\alpha bdbdbdp[1-p\alpha(1+bd+bdbd)]^{-2}bac\\
\end{array}$$ It is easy to check that
$$\begin{array}{ll}
&bacacd\alpha bdbdbdp[1-p\alpha(1+bd+bdbd)]^{-2}\\
=&bdbdbd\alpha bdbdbdp[1-p\alpha(1+bd+bdbd)]^{-2}\\
=&(bd)^6[1-p\alpha(1+bd+bdbd)]^{-2}(\alpha-\alpha^2\alpha^D)\\
\in &R^{nil}.
\end{array}$$ In view of \cite[Lemma 2.2]{L}, $\beta\beta^{\pi}\in R^{nil}$. Therefore $\beta\in R^D$ and $\beta^D=\beta^d$, as desired.

$\Longleftarrow$ This is obvious by the symmetry and Jacobson's lemma for Drazin inverse.\end{proof}

\begin{cor} Let $R$ be a ring, and let $a,b,c,d\in R$ satisfying
$$\begin{array}{c}
b(ac)^2=b(ac)(db)=b(db)(ac)=b(db)^2;\\
c(ac)^2=c(ac)(db)=c(db)(ac)=c(db)^2.
\end{array}$$
Then $\beta=1-ac\in R^{\#}$ if and only if $\alpha=1-bd\in R^{\#}$ and
$\beta^{\#}=\big[1-acd\alpha^{\pi}bac\big](1+ac+acac)+acd\alpha^{\#}bac$.\end{cor}
\begin{proof}  $\Longrightarrow $ Since $\beta\in R^{\#}$, we see that $\alpha \alpha^{\#}=0$. In view of Theorem 4,4, we have
$\beta^{D}=\big[1-acd\alpha^{\pi}bac\big](1+ac+acac)+acd\alpha^{\#}bac$.
As in the proof of Theorem 4.4, we get $$\beta\beta^{\pi}=acd\alpha (bd)^3\alpha^{\pi}[1-p\alpha(1+bd+bdbd)]^{-2}bac=0,$$ and therefore
$\beta^{\#}=\beta^D$, as required.

$\Longleftarrow $ This is symmetric.\end{proof}

\begin{rem} \end{rem} Many authors investigated the common spectral properties for bounded linear operators over Banach spaces (see~\cite{R2,YZ2,Z3,Z}). For bounded linear operators $A,B,C$ and $D$, we note that $I-AC$ (resp. $AC$) and $I-BD$ (resp. $BD$) share the common spectral properties under our preceding conditions.


\vskip10mm\begin{thebibliography}{99}\bibitem{CS1} H. Chen and M.S. Abdolyousefi, Generalized Jacobson's Lemma in a Banach algebra,
{\it Comm. Algebra}, {\bf 49}(2021), 3263--3272.

\bibitem{CS2} H. Chen and M. Sheibani, Cline's formula for g-Drazin inverses in a ring,
{\it Filomat}, {\bf 33}(2019), 2249--2255.

\bibitem{C} G. Corach, Extensions of Jacobson's lemma, {\it Comm. Algebra},
{\bf 41}(2013), 520-531.

\bibitem{CH} D. Cvetkovic-Ili\'c and R. Harte, On Jacobson's lemma and Drazin invertibility, {\it Applied Math. Letters}, {\bf 23}(2010),
417--420.

\bibitem{L} Y. Lian and Q. Zeng, An extension of Cline's formula for generalized Drazin inverse, {\it Turk. Math. J.},
{\bf 40}(2016), 161--165.

\bibitem{LC} Y. Liao; J. Chen and J. Cui, Cline's formula for the generalized Drazin inverse, {\it Bull. Malays. Math. Sci. Soc.},
{\bf 37}(2014), 37--42.

\bibitem{Mi} V.G. Miller and H. Zguitti, New extensions of Jacobson's lemma and Cline's formula, {\it Rend. Circ. Mat. Palermo, II. Ser.},
{\it 67}(2018), 105--114.

\bibitem{M1} D. Mosi\'c, A note on Cline's formula for the generalized Drazin inverse, {\it Linear Multilinear Algebra},
{\bf 63}(2014), 1106--1110.

\bibitem{M} D. Mosi\'c, Extensions of Jacobson's lemma for Drazin inverses, {\it Aequat. Math.}, {\bf 91}(2017), 419--428.

\bibitem{M2} D. Mosi\'c, On Jacobson's lemma and Cline's formula for Drazin inverses, , {\it Revista de la Uni\'{o}n Matem\'{a}tica Argentina},
{\bf 61}(2020), 267--276.

\bibitem{R} Y. Ren and L. Jiang, Extensions of Jacobson's lemma for generalized inverses in a ring, arXiv:2104.09846v3 [math.RA] 22 Apr 2021.  \url{https://arxiv.org/abs/2104.09846}.
    
    \bibitem{R2} Y. Ren and L. Jiang, Remark on common properties of the products $ac$ and $ba$, arXiv:2104.11862v1 [math.RA] 24 Apr 2021.  \url{https://arxiv.org/abs/2104.11862v1}.

\bibitem{YF} K. Yang and Q. Zeng, The generalized inverses of the products of two elements in a ring,
{\it Turk. J. Math.}, {\bf 44}(2020), 1744--1756.

\bibitem{YZ} K. Yan; Q. Zeng and Y. Zhu, Generalized Jacobson's
lemma for Drazin inverses and its applications, {\it Linear and Multilinear Algebra}, {\bf 68}(2020), 81--93.

\bibitem{YZ2} K. Yan; Q. Zeng and Y. Zhu, On Drazin spectral equation for the operator products, {\it Complex Analysis and Operator Theory}, {\bf 14}(2020),
\url{https://doi.org/10.1007/s11785-019-00979-y}.

\bibitem{Z2} Q.P. Zeng; Z. Wu and Y. Wen, New extensions of Cline's formula for generalized inverses, {\it Filomat},
{\bf 31}(2017), 1973--1980.

\bibitem{Z3} Q.P. Zeng; K. Yan and S. Zhang, New results on common properties of the products $AC$ and $BA$, II, {\it Mathematische Nachrichten}, {\bf 293}(2020), 1629--1635.

\bibitem{Z} H. Zguitti, A note on the common spectral properties for bounded linear operators,
{\it Filomat}, {\bf 33}(2019), 4575--4584.

\bibitem{Z3} G. Zhuang; J. Chen and J. Cui, Jacobson's lemma for the generalized Drazin inverse, {\it Linear Algebra Appl.},
{\bf 436}(2012), 742--746.

\end{thebibliography}
\end{document}